\documentclass[12pt]{article}
\usepackage{amsmath, amsthm}\usepackage{enumerate}\usepackage{amssymb}\usepackage{bbold}
\usepackage{bbm}\usepackage{dsfont}\usepackage{color}
\newtheorem{theorem}{Theorem}[section]

\newtheorem{lemma}[theorem]{Lemma}

\newtheorem{proposition}[theorem]{Proposition}

\theoremstyle{definition}

\newtheorem{definition}[theorem]{Definition}
\newtheorem{example}[theorem]{Example}

\newtheorem{remark}[theorem]{Remark}

\newtheorem{question}[theorem]{Question}

\DeclareMathOperator{\stmuc}{\xrightarrow[]{st_\mu}}
\DeclareMathOperator{\stc}{\xrightarrow[]{\mu-st_\tau}}

\begin{document}

\title{Statistical convergence of nets on locally solid Riesz spaces}\maketitle\author{\centering{Fatih Tem\i zsu$^{1}$, {Abdullah Ayd{\i}n$^{2,*}$\\ \small $^1$Department of Mathematics, Bingöl University, Bingöl, Turkey\\ ftemizsu@bingol.edu.tr\\ \small $^2$Department of Mathematics, Mu\c{s} Alparslan University, Mu\c{s}, Turkey\\ a.aydin@alparslan.edu.tr\\ $*$Corresponding Author}
		
\abstract{The statistical convergence is handled for sequences with the natural density, in general. In a recent paper, the statistical convergence for nets in Riesz spaces has been studied and investigated by developing topology-free techniques in Riesz spaces. In this paper, we introduce the statistically topological convergence for nets on locally solid Riesz spaces with solid topologies. Moreover, we introduce the statistical continuity on locally solid Riesz spaces.}\\
\vspace{2mm}

\vspace{2mm}

\section{Introduction and preliminaries}

Riesz space and statistical convergence are the natural and efficient tools in the theory of functional analysis. Riesz space was introduced by F. Riesz in \cite{Riez} and the idea of statistical convergence was firstly introduced by Zygmund \cite{Zygmund}, after then, Fast \cite{Fast} and Steinhaus \cite{St} independently improved the idea of statistical convergence. Riesz space is an ordered vector space having many applications in measure theory, operator theory, and applications in economics (cf. \cite{AB,ABPO,Za}). On the other hand, statistical convergence is a generalization of the ordinary convergence of a real sequence. Several applications and generalizations of the statistical convergence of sequences have been investigated by several authors (cf. \cite{Aydn1,Aydn2,Fast,Fridy,M,St,SP}). In general, the statistical convergence of sequences is considered with the natural density of sets on the natural numbers $\mathbb{N}$. However, Connor introduced the notion of statistical convergence for sequences with finitely additive set function \cite{Con,Con2}. After then, some similar works have been done (cf. \cite{CLL,DO,Mor}). The study related to this paper is done by Ayd\i n and Temizsu in \cite{AT}, where the statistical convergence was introduced for nets. In this work, we
introduce the concept of statistical convergence for nets and statistical continuous operators on locally solid Riesz space with solid topologies. 

First, let us remember some notations and terminologies used in this paper. A binary relation \textquotedblleft$\leq$\textquotedblright \ on a set $A$ is called a {\em preorder} if it is reflexive and transitive. A non-empty set $A$ with a preorder binary relation "$\leq$" is said to be a \textit{directed upwards} (\textit{directed set}, shortly) if for each pair $x,y\in A$ there exist $z\in A$ such that $x\leq z$ and $y\leq z$. Unless otherwise stated, we consider all directed sets as infinite in this paper. For given elements $a$ and $b$ in a preordered set $A$ such that $a\leq b$, the set $\{x\in A:a\leq x\leq b\}$ is called an \textit{order interval} in $A$. A subset $I$ of $A$ is called an \textit{order bounded set} whenever $I$ is contained in an order interval. A function whose domain is a directed set is said to be a \textit{net}. A net is briefly abbreviated as $(x_\alpha)_{\alpha\in A}$ with its directed domain set $A$. 

We remind that a map from a field $\mathcal{M}$ to $[0,\infty]$ is called {\em finitely additive measure} whenever $\mu(\emptyset)=0$ and $\mu(\cup_{i=1}^{n}E_i)=\sum_{i=1}^{n}\mu(E_i)$ for all finite disjoint sets $\{E_i\}_{i=1}^n$ in $\mathcal{M}$ (cf. \cite[p.25]{Folland}). We take the following definitions from \cite{AT}.
\begin{definition}
Let $A$ be a directed set and $\mathcal{M}$ be a subfield of $\mathcal{P}(A)$. 
\begin{enumerate}
\item[(1)] \ An order interval $[a,b]$ of $A$ is said to be a {\em finite order interval} if it is a finite subset of $A$.
\item[(2)] \ $\mathcal{M}$ is called an {\em interval field} on $A$ whenever it includes all finite order intervals of $A$.
\item[(3)] \ A finitely additive measure $\mu:\mathcal{M}\to[0,1]$ is said to be a {\em directed set measure} if $\mathcal{M}$ is an interval field and $\mu$ satisfies the following facts: $\mu(I)=0$ for every finite order interval $I\in \mathcal{M}$; $\mu(A)=1$; $\mu(C)=0$ whenever $C\subseteq B$ and $\mu(B)=0$ holds for $B,C\in \mathcal{M}$.
\end{enumerate} 
\end{definition} 

Following from \cite[Rem.2.5]{AT} that the directed set measure is an extension of the natural density. In this paper, we consider all nets with a directed set measure on the interval fields of the power set of  the index sets. Moreover, to simplify the presentation, a directed set measure on an interval field $\mathcal{M}$ of directed set $A$ will be expressed briefly as a measure on the directed set $A$.

Recall that a real vector space $E$ with an order relation \textquotedblleft$\leq$\textquotedblright \ is called an {\em ordered vector space} if, for each $x,y\in E$ with $x\leq y$, $x+z\leq y+z$ and $\alpha x\leq\alpha y$ hold for all $z\in E$ and $\alpha \in \mathbb{R}_+$. An ordered vector space $E$ is called a {\em Riesz space} or {\em a vector lattice} if, for any two vectors $x,y\in E$, the infimum and the supremum
$$
x\wedge y=\inf\{x,y\} \ \ \text{and} \ \ x\vee y=\sup\{x,y\}
$$
exist in $E$, respectively.	A subset $I$ of a Riesz space $E$ is said to be a \textit{solid} set if, for each $x\in E$ and $y\in I$ with $|x|\leq|y|$, it follows that $x\in I$. A solid vector subspace is called an \textit{order ideal}. An order closed ideal is called a \textit{band}. Also, a band $B$ is called a \textit{projection band} whenever it satisfies $E=B\oplus B^d$, where $B^d$ is the disjoint complement set of $B$. A Riesz space $E$ has the \textit{Archimedean} property provided that $\frac{1}{n}x\downarrow0$ holds in $E$ for each $x\in E_+$. In this paper, unless otherwise stated, all Riesz spaces are assumed to be real and Archimedean. We continue with the crucial notion of Riesz spaces (cf. \cite{AB,ABPO,LZ,Za}).

\begin{definition}
A net $(x_\alpha)_{\alpha\in A}$ in a Riesz space $E$ is called {\em order convergent} to $x\in E$ if there exists another net $(y_\alpha)_{\alpha\in A}\downarrow 0$ (i.e., $\inf y_\alpha=0$ and $y_\alpha\downarrow$) such that $|x_\alpha-x|\le y_\alpha$ holds for all $\alpha\in A$.
\end{definition}

We refer the reader for some different types of the order convergence and some relations among them to \cite{AS}. 

We remind that a {\em linear topology} $\tau$ on a vector space $E$ means that it is a topology on $E$ which makes the addition and the scalar multiplication continuous. For each topological vector space, it is well known that there is a base $\mathcal{N}$ consisting of zero neighborhoods holding the following properties (cf. \cite{AB,AlTo}): 
\begin{enumerate}
\item[(a)] \ Every $U\in \mathcal{N}$ is a \textit{balanced} set, i.e., $\lambda U\subseteq U$ for all $\lvert\lambda\rvert\leq 1$;
\item[(b)] \ Each $U\in \mathcal{N}$ is an \textit{absorbing set}, i.e., for every $u\in U$, there exists $\lambda>0$ such that $\lambda u\in U$;
\item[(c)] \ For every $U\in \mathcal{N}$, there is another $V\in \mathcal{N}$ with $V+V\subseteq U$;
\item[(d)] \ For any $U_1,U_2\in \mathcal{N}$, there is $U\in \mathcal{N}$ such that $U\subseteq U_1\cap U_2$; 
\item[(e)] \ For every $U\in\mathcal{N}$ and every scalar $\lambda$, the set $\lambda U$ is in $\mathcal{N}$.
\end{enumerate}
Whenever we mention a basic zero neighborhood, we always assume that it belongs to a base which satisfies the properties $(a)$-$(e)$.
\begin{definition}
Let $\tau$ be a linear topology on a Riesz space $E$. Then $(E,\tau)$ is said to be a {\em locally solid Riesz space} or {\em locally solid vector lattice} whenever $\tau$ has a base at zero which consists of solid sets.
\end{definition}
We denote $\mathcal{N}_{sol}$ as the solid base of a locally solid Riesz space.


\section{Statistically topological convergence}

Let $K$ be a subset of natural numbers and define a new set $K_n=\{k\in K:k\leq n\}$. Then we denote $\lvert K_n\rvert$ for the cardinality of that the set $K_n$. If the limit $\delta(K):=\lim\limits_{n\to\infty}\lvert K_n\rvert/n$ exists then $\delta(K)$ is called the \textit{natural density} or \textit{asymptotic density} of the set $K$. On the other hand, let $X$ be a topological space and $(x_n)$ be a sequence in $X$. Then $(x_n)$ is said to be statistically convergent to $x\in X$ whenever, for each neighborhood $U$ of $x$, we have $\delta\big(\{n\in\mathbb{N}:x_n\notin U\}\big)=0$ (cf. \cite{M,MK}). Motivated by the above definitions, we give the following notion which is crucial for the present paper.
\begin{definition}
Let $(x_\alpha)_{\alpha\in A}$ be a net in a locally solid Riesz space $(E,\tau)$. Then $(x_\alpha)_{\alpha\in A}$ is called {\em $\mu$-statistically topological convergent} to $x\in E$ if, for every zero $\tau$-neighborhood $U$, there exists an index $\alpha_u$ related $U$ such that 
$$
\mu\big(\{\alpha_u \leq \alpha\in A:(x_\alpha-x)\notin U\}\big)=0.
$$
We abbreviate this convergence as $x_\alpha\stc x$. Then we shortly say that $(x_\alpha)_{\alpha\in A}$ is $\mu$-statistically $\tau$-convergent to $x$. 
\end{definition}

Briefly, $x_\alpha\stc x$ if $\mu\big(B_{\{\alpha_u,U\}}\{x_\alpha,x\}\big)=0$ for each zero $\tau$-neighborhood $U$ and for some indexes $\alpha_u\in A$, where
$$
B_{\{\alpha_u,U\}}=\{\alpha_u\leq\alpha\in A:(x_\alpha-x)\notin U\}.
$$
We denote $E_{\mu-st_\tau}$ as the set of all $\mu$-statistically $\tau$-convergent nets in a locally solid Riesz space $(E,\tau)$. Also, we observe the following useful and important fact.
\begin{remark}\label{main remark}
Consider a net $(x_\alpha)_{\alpha\in A}$ with $x_\alpha\stc x$ in a locally solid Riesz space $(E,\tau)$. Then, for a fixed zero $\tau$-neighborhood $U$, there exists an index $\alpha_u \in A$ such that $\mu\big(B_{\{\alpha_u,U\}}\{x_\alpha,x\}\big)=0$, where
\begin{eqnarray*}
B_{\{\alpha_u,U\}}&=&\{\alpha_u\leq\alpha\in A:(x_\alpha-x)\notin U\}\\&=& \{\alpha\in A:\alpha_u\leq\alpha\}\cap \{\alpha\in A:(x_\alpha-x)\notin U\}.
\end{eqnarray*}
Thus, the complement set of $B_{\{\alpha_u,U\}}$ is
$$
B^c_{\{\alpha_u,U\}}=\{\alpha\in A:\alpha_u\nleq\alpha\}\cup \{\alpha\in A:(x_\alpha-x)\in U\}.
$$
So, it follows that $\mu(B^c)=1$, and so, one can obtain that 
$$
\mu\big(\{\alpha\in A:(x_\alpha-x)\in U\}\big)=1
$$
does need not hold in general. Therefore, generally, we have $\mu\big(\{\alpha\in A:\alpha_u\nleq\alpha\}\big)\neq 0$.
\end{remark}

We can introduce the following notion: a net $x_\alpha\stc x$ is said to be {\em straight $\mu$-statistically topological convergent} to $x\in E$ if $\mu\big(\{\alpha\in A:(x_\alpha-x)\in U\}\big)=1$, where $\alpha_u\in A$ is the index of $\mu$-$st_\tau$-convergence. But, we do not dwell on this definition in this study. Hence, unless otherwise stated, we assume that the measure of sets $\{\alpha\in A:\alpha_u\nleq\alpha\}$ and $\{\alpha\in A:(x_\alpha-x)\in U\}$ in Remark \ref{main remark} are different from zero for all $\mu$-statistically $\tau$-convergent nets $(x_\alpha)_{\alpha\in A}$ and for each zero $\tau$-neighborhood $U$. Following from the solidness property of zero $\tau$-neighborhoods, we give the following observation.
\begin{lemma}
It is clear that $x_\alpha\stc x$ if and only if $|x_\alpha-x|\stc 0$ in locally solid Riesz spaces.
\end{lemma}

The following example which is similar to \cite[Exam.1.3]{Aydn1} yields two sequences, one is $\mu$-statistically $\tau$-convergent and the other is not on the same space.

\begin{example}
Let consider the Riesz space $E:=c_0$ the set of all real null sequences. Thus, $E$ is a Banach lattice with the supremum norm $\lVert\cdot\rVert_\infty$. It follows from \cite[Thm.2.28]{AB} that $(E,\lVert\cdot\rVert_\infty)$ is also a locally solid Riesz space. It can be seen that the base of solid topology generated by the supremum norm consists of zero neighborhoods 
$$
W_\varepsilon=\{x\in E:\lVert x\rVert_\infty<\varepsilon\},
$$
where $\varepsilon$ is an arbitrary positive real number. Now, consider the sequence $(e_n)$ of the standard unit vectors in $E$. Then $e_n\not\xrightarrow{\mu-st_\tau}0$ in $E$. Indeed, take an arbitrary zero $\tau$-neighborhood $U$. Then there exists some $W_\varepsilon\in \mathcal{N}_{sol}$ for some $\varepsilon>0$ such that $W_\varepsilon\subseteq U$. Thus, it follows that
$$
B=\{n\in \mathbb{N}:e_n\notin W_\varepsilon\}=\{n\in \mathbb{N}:\lVert e_n\rVert_\infty\geq \varepsilon\}.
$$
It can be seen that $\mu(B)\neq 0$ because $\lVert e_n\rVert_\infty=1>\varepsilon$ holds for all $n\in\mathbb{N}$. Therefore, we obtain the desired result.

Now, take another sequence $(x_n)$ in $E$ which is denoted by 
$$
x_n:=(0,0,0,\cdots,0,\frac{1}{n},0,\cdots)
$$
for each $n\in\mathbb{N}$. Thus, we have $e_n\stc0$. To see this, consider an arbitrary zero $\tau$-neighborhood $U$ with a solid zero $\tau$-neighborhood $W_\varepsilon\subseteq U$. Then there exists a natural number $n_0$ such that $1/{n_0}\leq \varepsilon$. So, it follows that $W_{\frac{1}{n_0}}\subseteq W_\varepsilon$. Also, we have
$$
C=\{n\in \mathbb{N}:x_n\notin W_\frac{1}{n_0}\}=\{n\in \mathbb{N}:\lVert x_n\rVert_\infty=\frac{1}{n}>\frac{1}{n_0}\}.
$$
Thus, we obtain $\mu(C)=0$. Since $\{n\in \mathbb{N}:x_n\notin U\}\subseteq\{n\in \mathbb{N}:x_n\notin W_\varepsilon\}\subseteq C$, we have $\mu\big(\{n\in \mathbb{N}:x_n\notin U\}\big)=0$. Therefore, we obtain the desired result, $x_n\stc0$.
\end{example}


\section{Results of $\mu$-statistically topological convergence}
Recall that a net $(x_\alpha)_{\alpha\in A}$ topological converges to a point $x$ in a topological space $X$ if, for every neighborhood $U$ of $x$, there is an index $\alpha_0\in A$ such that $x_\alpha \in U$ for all $\alpha\geq\alpha_0$. 
\begin{remark}\label{top con imp mu}
The topological convergence implies the $\mu$-statistically topological convergence. Indeed, suppose that $x_\alpha\xrightarrow{\tau}x$ in a locally solid Riesz space $(E,\tau)$. Then, for arbitrary zero $\tau$-neighborhood $U$, we have an index $\alpha_u$ such that $(x_\alpha-x)\in U$ for all $\alpha\geq\alpha_u$. Thus, we get $\mu\big(\{\alpha_u\leq\alpha\in A: (x_\alpha-x)\not\in U\}\big)=0$, i.e., $x_\alpha\stc x$.
\end{remark}

The converse of Remark \ref{top con imp mu} does not need to be true. We continue with the following several basic and useful results which are similar to the classical ones for so many kinds of statistical convergences.
\begin{theorem}\label{basic remarkasic properties of st convergence}
Let $x_\alpha\stc x$ and $y_\alpha\stc y$ in a locally solid Riesz space $(E,\tau)$. Then we have the following statements:
\begin{enumerate}
\item[(i)] \ if $\tau$ is Hausdroff, $x_\alpha\stc x$ and $x_\alpha\stc z$ then $x=z$;
\item[(ii)] \ $x_\alpha+y_\alpha\stc x+y$;
\item[(iii)] \ $\lambda x_\alpha\stc\lambda x$ for any $\lambda\in{\mathbb R}$;
\item[(iv)] \ $x_\alpha\stc x$ if and only if  $(x_n-x)\stc0$.		
\end{enumerate}	
\end{theorem} 

\begin{proof}
The $(iv)$ is straightforward, and so, we show the other statements. Let $U$ be an arbitrary zero $\tau$-neighborhood. Then there exists $W\in\mathcal{N}_{sol}$ such that $W \subseteq U$. Moreover, there is $V\in\mathcal{N}_{sol}$ so that $V+V\subseteq W$, and so, $V+V\subseteq U$.

$(i)$ It follows from $x_\alpha\stc x$ and $x_\alpha\stc z$ that there exist indexes $\alpha_1$ and $\alpha_2$ such that 
$$
\mu\big(B_{\{\alpha_1,V\}}\{x_\alpha,x\}\big)=\mu\big(B_{\{\alpha_2,V\}}\{x_\alpha,z\}\big)=0.
$$
Also, since the index set $A$ of the net $(x_\alpha)_{\alpha\in A}$ is directed, there exists $\alpha_0\in A$ such that $\alpha_1\leq\alpha_0$ and $\alpha_2\leq\alpha_0$. So, we have $\mu\big(\{\alpha_0\leq\alpha\in A:(x_\alpha-x)\notin V\}\big)=0$ and $\mu\big(\{\alpha_0\leq\alpha\in A:(x_\alpha-z)\notin V\}\big)=0$. Thus, $x_\alpha-x,x_\alpha-z\in V$ for some $\alpha\in A$. It follows that
$$
x-z=(x-x_\alpha)+(x_\alpha-z)\in V+V\subseteq U
$$
for some $\alpha\in A$. Hence, we get $(x-z)\in U$ for each zero $\tau$-neighborhood $U$. It is well known that the intersection of all zero $\tau$-neighborhood in Hausdorff space is the singleton zero. It means that $x=z$.

$(ii)$ There are some indexes $\alpha_1,\alpha_2\in A$ such that 
$$
\mu\big(B_{\{\alpha_1,V\}}\{x_\alpha,x\}\big)=\mu\big(B_{\{\alpha_2,V\}}\{y_\alpha,y\}\big)=0
$$
because of $x_\alpha\stc x$ and $y_\alpha\stc y$. Then there is $\alpha_0\in A$ such that $\alpha_1\leq\alpha_0$ and $\alpha_2\leq\alpha_0$. So, we have $\mu\big(B_{\{\alpha_0,V\}}\{x_\alpha,x\}\big)=\mu\big(B_{\{\alpha_0,V\}}\{y_\alpha,y\}\big)=0$. Hence, it follows from the equality
$$
(x_\alpha+y_\alpha)-(x+y)=(x_\alpha-x)+(y_\alpha+y)\in V+V\subseteq U
$$
that $(x_\alpha+y_\alpha)-(x+y)\notin U$ implies $(x_\alpha-x)\notin V$ or $(y_\alpha-y)\notin V$ for each $\alpha\in A$. Without loss of generality, assume that $(x_\alpha+y_\alpha)-(x+y)\notin U$ implies $(x_\alpha-x)\notin V$. Thus, we have
$$
\{\alpha\in A:(x_\alpha+y_\alpha)-(x+y)\notin U\}\subseteq \{\alpha\in A:x_\alpha-x\notin V\}.
$$
Therefore, we obtain $\mu\big(\{\alpha_0\leq\alpha\in A:(x_\alpha+y_\alpha)-(x+y)\notin U\}\big)=0$, i.e., $x_\alpha+y_\alpha\stc x+y$.

$(iii)$ Take a scalar $\lambda\in \mathbb{R}$ with $|\lambda|<1$. Then $(x_\alpha-x)\in W$ implies $\lambda(x_\alpha-x)=(\lambda x_\alpha-\lambda x)\in W$ because $W$ is a balanced set. Hence, one can see that
$$
\{\alpha_w\leq\alpha\in A:\lambda (x_\alpha-\lambda x)\notin W\}\subseteq B_{\{\alpha_w,W\}}.
$$
So, we obtain that $\mu\big(\{\alpha_w\leq\alpha\in A:(\lambda x_\alpha-\lambda x)\notin W\}\big)=0$. It follows $\mu\big(\{\alpha_w\leq\alpha\in A:(\lambda x_\alpha-\lambda x)\notin U\}\big)=0$ because of $W\subseteq U$. Therefore, we obtain that $\lambda x_\alpha\stc\lambda x$ for all $|\lambda|<1$. 

Next, choose a scalar $\lvert \lambda\rvert> 1$. For given $W$, another zero $\tau$-neighborhood solid set $N\in\mathcal{N}_{sol}$ can be found so that 
$$
\{N+N+\cdots+N\}_m\subseteq W
$$
holds for $m\in\mathbb{N}$, the smallest natural number greater or equal $|\lambda|$. It follows from $x_\alpha\stc x$ that $\mu\big(B_{\{\alpha_N,N\}}\{x_\alpha,x\}\big)=0$ for some $\alpha_N\in A$. We observe that
$$
|\lambda x_\alpha-\lambda x|=|\lambda||x_\alpha-x|\leq m|x_\alpha-x|\in N+\dots+N\subseteq W\subseteq U
$$
for some $\alpha\in A$. Then $(\lambda x_\alpha-\lambda x)\notin U$ implies $(x_\alpha-x)\notin N$ for $\alpha\in A$. By the same argument in the last part of proof $(ii)$, we have $\mu\big(\{\alpha_N\leq\alpha\in A:(\lambda x_\alpha-\lambda x)\notin U\}\big)=0$, i.e., $\lambda x_\alpha\xrightarrow{st-u_\tau}\lambda x$. 
\end{proof}

\begin{definition}\label{bounded def}
A net $(x_\alpha)_{\alpha\in A}$ in a locally solid Riesz space $(E,\tau)$ is called {\em $\mu$-statistically $\tau$-bounded} whenever there is some $\lambda>0$ scalar such that
$$
\mu\big(\{\alpha_u\leq\alpha\in A:\lambda x_\alpha\notin U\}\big)=0
$$
holds for each zero $\tau$-neighborhood $U$ with some index $\alpha_u\in A$.
\end{definition}

One can see that a topological bounded net is $\mu$-statistically $\tau$-bounded in locally solid Riesz spaces. However, the converse not need to be true. To see this, we consider \cite[Exam.2.3]{AT}.
\begin{example}
Take a field $\mathcal{M}$ consisting of countable or co-countable subsets of the index set of a net $(x_\alpha)_{\alpha\in A}$. Then every subnet of $(x_\alpha)_{\alpha\in A}$ with the countable index is $\mu$-statistically $\tau$-bounded. But $(x_\alpha)_{\alpha\in A}$ might have a subnet with a countable index that is not topologically bounded.
\end{example}

\begin{remark} It follows from \cite[Thm.2.19]{AB} every order bounded set is topologically bounded. Thus, every ordered bounded net is $\mu$-statistically $\tau$-bounded in locally solid Riesz spaces.
\end{remark}

\begin{theorem}
Every $\mu$-statistically $\tau$-convergent net is $\mu$-statistically $\tau$-bounded.
\end{theorem}

\begin{proof}
Assume that a net $(x_\alpha)_{\alpha\in A}$ is $\mu$-statistically $\tau$-converges to $x$ in a locally solid Riesz space $(E,\tau)$. Fix a zero $\tau$-neighborhood $U$. Then there are $W,V\in \mathcal{N}_{sol}$ such that $V+V\subseteq W\subseteq U$. So, there is an index $\alpha_v$ such that $\mu\big(B_{\{\alpha_v,V\}}\{x_\alpha,x\}\big)=0$. By using the absorbing property of $V$, there is a scalar $\beta>0$ such that $\beta x\in V$. Now, choose a scalar $\sigma_u$ such that $|\sigma_u|\leq1$ and $|\sigma_u|\leq\beta$. Then it follows from $|\sigma_u x|\leq|\beta x|$ that $\sigma_u x\in V$. Moreover, $(x_\alpha-x)\in V$ implies $\sigma_u(x_\alpha-x)\in V$ because $V$ is balanced. Then we observe from
$$
\sigma_u x_\alpha=\sigma_u(x_\alpha-x)+\sigma_u x\in V+V\subset U
$$
that $\sigma_u x_\alpha\notin U$ implies $\sigma_u(x_\alpha-x)\notin V$ for $\alpha\in A$. Hence, we obtain 
\begin{eqnarray*}
\{\alpha\in A:\sigma_u x_\alpha\notin U\}&\subseteq& \{\alpha\in A:\sigma(x_\alpha-x)\notin V\}\\ &\subseteq& \{\alpha\in A:(x_\alpha-x)\notin V\}.
\end{eqnarray*}
It follows that $\mu\big(\{\alpha_v\leq\alpha\in A:\sigma_u x_\alpha\notin U\}\big)=0$. If we take the scalar $\lambda$ in Definition \ref{bounded def} as the minimum of $\{\sigma_u:U\ \text{is zero} \ \tau\text{-neighborhood}\}$ then we get the desired result.
\end{proof}

\begin{definition}
A net $(x_\alpha)_{\alpha\in A}$ in a locally solid Riesz space $(E,\tau)$ is called {\em $\mu$-statistically $\tau$-Cauchy} if the net $(x_\alpha-x_{\alpha'})_{(\alpha,\alpha')\in A\times A}$ is $\mu$-statistically $\tau$-convergent to zero.
\end{definition}

\begin{question}
Is $\mu$-statistically $\tau$-Cauchy net $\mu$-statistically $\tau$-bounded?
\end{question}

\begin{theorem}
Every $\mu$-statistically $\tau$-convergent net is a $\mu$-statistically $\tau$-Cauchy.
\end{theorem}

\begin{proof}
Suppose that $(x_\alpha)_{\alpha\in A}$ is $\mu$-statistically $\tau$-convergent to $x$ in a locally solid Riesz space $(E,\tau)$. Then, for any zero $\tau$-neighborhood $U$ with zero $\tau$-neighborhood solid sets $V+V\subseteq W\subseteq U$, there exists an index $\alpha_v$ such that $\mu\big(B_{\{\alpha_v,V\}}\{x_\alpha,x\}\big)=0$. Then we have
$$
x_\alpha-x_{\alpha'}=(x_\alpha-x)+(x-x_{\alpha'})\in V+V\subseteq U
$$
for some $\alpha,\alpha'\in A$. Thus, one can see that $(x_\alpha-x_{\alpha'})\notin U$ implies $(x_\alpha-x)\notin V$ for some $\alpha$. Thus, we obtain
$$
\mu\big(\{\alpha_v\leq\alpha\in A:(x_\alpha-x_{\alpha'})\notin U\}\big)=0.
$$
Therefore, $(x_\alpha-x_{\alpha'})_{(\alpha,\alpha')\in A\times A}\stc 0$. It proves that $(x_\alpha)_{\alpha\in A}$ is $\mu$-statistically $\tau$-Cauchy net.
\end{proof}

Recall that a linear topology $\tau$ on a Riesz space $E$ is locally solid if and only if it is generated by a family $\{\rho_j\}_{j\in J}$ of Riesz pseudonorms (cf. \cite[Thm.2.28]{AB}). It is well known that if a subset $A$ of $E$ is topological bounded then $\rho_j(A)$ is bounded in $\mathbb{R}$ for each $j\in J$.
\begin{remark}\label{remarkkk} If a net $(x_\alpha)_{\alpha\in A}$ is $\mu$-statistically $\tau$-bounded in a locally solid vector lattice with a family of a Riesz pseudonorms $\{\rho_j\}_{j\in J}$ then $\rho_j(x_\alpha)$ is not need to be bounded in $\mathbb{R}$ for all $j\in J$.
\end{remark}

The converse of Remark \ref{remarkkk} is also not need to be hold. To see this, we consider the following example.
\begin{example}
Let $(E,\tau)$ be a locally solid vector lattice with a family of a Riesz pseudonorms $\{\rho_j\}_{j\in J}$. Define $\hat{\rho_j}:=\frac{\rho_j}{1+\rho_j}$ for each $j\in J$. Then it can be seen that $\hat{\rho_j}$ is also a Riesz pseudonorm on $E$ and the topology $\tau$ is generated by the family of $\{\hat{\rho_j}\}_{j\in J}$. Moreover, $\hat{\rho_j}(x_\alpha)\leq 1$ for all $\alpha\in A$. It means that $(x_\alpha)_{\alpha\in A}$ is bounded in $\mathbb{R}$ for every net $(x_\alpha)_{\alpha\in A}$ in $E$. However, we might have a net in $E$ that is not $\mu$-statistically $\tau$-bounded.
\end{example}

\begin{proposition}\label{normality of st conv}
Let $(x_\alpha)_{\alpha\in A}$, $(y_\alpha)_{\alpha\in A}$ and $(z_\alpha)_{\alpha\in A})$ nets in a locally solid Riesz space $(E,\tau)$ with $x_\alpha\leq y_\alpha\leq z_\alpha$ for all $\alpha\in A$. Then $x_\alpha\stc w$ and $x_\alpha\stc w$ implies $y_\alpha\stc w$. 
\end{proposition}

\begin{proof}
Let $U$ be an arbitrary zero $\tau$-neighborhood in $E$ with $V\in \mathcal{N}_{sol}$ such that $V+V\subseteq U$. Since $x_\alpha\stc w$ and $z_\alpha\stc w$, there exist $\alpha_v$ such that
$$
\mu\big(B_{\{\alpha_v,V\}}\{x_\alpha,w\}\big)=\mu\big(B_{\{\alpha_v,V\}}\{z_\alpha,w\}\big)=0.
$$
On the other hand, it follows from $x_\alpha\leq y_\alpha\leq z_\alpha$ for all $\alpha\in A$ that
$$
|y_\alpha-w|\leq|x_\alpha-w|+|z_\alpha-w|\in V+V\subseteq U
$$
holds for some $\alpha\in A$. So, $(y_\alpha-w)\notin U$ implies that $(x_\alpha-w)\notin V$ or $(z_\alpha-w)\notin V$ for $\alpha\in A$. Therefore, one can get $\mu\big(\{\alpha_0\leq\alpha:(y_\alpha-w)\notin V\}\big)=0$ for both cases. Hence, $y_\alpha\stc w$. 
\end{proof}

\begin{theorem}
Let $C$ be a projection band and $x_\alpha\stc x$ in a locally solid Riesz space $(E,\tau)$. Then $P_C(x_n)\stc P_C(x)$ for the corresponding order projection $P_C$ of $C$. 
\end{theorem}

\begin{proof}
Consider a fixed zero $\tau$-neighborhood $U$. Then there is an index $\alpha_u$ such that $\mu\big(B_{\{\alpha_u,U\}}\{x_\alpha,x\}\big)=0$ because of $x_\alpha\stc x$. On the other hand, it follows from \cite[Thm.24.5 and Thm.24.6]{LZ} that $P_C$ is an order continuous lattice homomorphism and $0\leq P_C\leq I$. Also, by considering \cite[Thm.2.14]{ABPO}, we observe that
$$
\lvert P_C(x_\alpha)-P_C(x)\rvert=P_C(\lvert x_\alpha-x\rvert)\leq \lvert x_\alpha-x\rvert.
$$ 
holds for all $\alpha\in A$. Thus, $(P_C(x_\alpha)-P_C(x))\notin U$ implies $(x_\alpha-x)\notin U$ for all $\alpha\geq \alpha_u$. Thus, it follows that
$$
\mu\big(\{\alpha_u\leq\alpha\in A:(P_C(x_\alpha)-P_C(x))\notin U\}\big)=0.
$$
Therefore, $P_C(x_n)\xrightarrow{st-u_\tau} P_C(x)$ in $E$.
\end{proof}

\section{The $\mu$-statistically continuity}

\begin{definition}
An operator $T$ between locally solid Riesz spaces $(E,\tau)$ and $(F,\tau')$ is called {\em $\mu$-statistically topological continuous operator} whenever $x_\alpha\stc x$ in $E$ implies $T(x_\alpha)\stc T(x)$ in $F$.
\end{definition}

We show that the lattice operators are $\mu$-statistically topological continuous in the following sense.
\begin{theorem}\label{lattice operators}
Let $(x_\alpha)_{\alpha\in A}$ and $(y_\alpha)_{\alpha\in A}$ be two nets in a locally solid Riesz space $(E,\tau)$. If $x_\alpha\stc x$ and $y_\alpha\stc y$ then:
\begin{enumerate}
\item[(i)] $x_\alpha\wedge y_\alpha\stc x\wedge y$;
\item[(ii)] $x_\alpha\vee y_\alpha\stc x\vee y$;
\item[(iii)] $x^+_\alpha\stc x^+$;
\item[(iv)] $x^-_\alpha\stc x^-$;
\item[(v)] $|x_\alpha|\stc |x|$.
\end{enumerate}
\end{theorem}

\begin{proof}
It is enough to show statement $(ii)$ because one can obtain the other statements from \cite[Thm.1.7]{ABPO}.
	
Fix an arbitrary zero $\tau$-neighborhood $U$ in $E$ with $V\in\mathcal{N}_{sol}$ such that $V+V\subseteq U$. Thus, there exists an index $\alpha_0$ such that $\mu\big(B_{\{\alpha_0,V\}}\{x_\alpha,x\}\big)=\mu\big(B_{\{\alpha_0,V\}}\{y_\alpha,y\}\big)=0$ because of $x_\alpha\stc x$ and $y_\alpha\stc y$. On the other hand, we observe from \cite[Thm.12.4]{LZ} that
\begin{eqnarray*}
\lvert x_\alpha\vee y_\alpha-x\vee y\rvert&\leq& \lvert x_\alpha\vee y_\alpha-y_\alpha\vee x\rvert+\lvert y_\alpha\vee x-x\vee y\rvert\\&\leq& \lvert x_\alpha-x\rvert+\lvert y_\alpha-y\rvert
\end{eqnarray*}
holds for each $\alpha\in A$. So, $(x_\alpha\vee y_\alpha-x\vee y)\notin U$ implies that $(x_\alpha-x)\notin V$ or $(y_\alpha-y)\notin V$. Therefore, we get $\mu\big(\{\alpha_0\leq\alpha\in A:(x_\alpha\vee y_\alpha-x\vee y)\notin U\}\big)=0$ for two cases, and so, we get the desired result.  
\end{proof}

The positive cone of a Riesz space is denoted by $E_+:=\{x\in E:0\leq x\}$. Also, it follows from Theorem \ref{basic remarkasic properties of st convergence} and Theorem \ref{lattice operators} that $E_+$ is closed under the $\mu$-statistically $\tau$-convergence in Hausdorff locally solid Riesz spaces.

\begin{proposition}\label{motone and st implies order}
Every monotone $\mu$-statistically $\tau$-convergent net in Hausdorff locally solid Riesz spaces is order convergent.
\end{proposition}

\begin{proof}
Suppose that a net $(x_\alpha)_{\alpha\in A}$ is increasing and $x_\alpha\stc x$ in a Hausdorff locally solid Riesz space $(E,\tau)$. We show that $x_\alpha\uparrow x$. Take an arbitrary zero $\tau$-neighborhood $U$. Then there exists an index $\alpha_0$ such that $\mu\big(B_{\{\alpha_u,U\}}\{x_\alpha,x\}\big)=0$. So, we obtain
$$
\mu\big(\{\alpha_u\leq\alpha\in A:(x_\alpha-x_{\alpha_u})-(x-x_{\alpha_u})\notin U\}\big)=0.
$$
Hence, $x_\alpha-x_{\alpha_u}\stc x-x_{\alpha_u}$. On the other hand, it follows from the increasing of $(x_\alpha)_{\alpha\in A}$ that $x_\alpha-x_{\alpha_u}\in E_+$ for all $\alpha\geq\alpha_u$. Thus, $x\geq x_{\alpha_u}$. It means that $x$ is an upper bound of $(x_\alpha)_{\alpha\in A}$. Take another upper bound $z$ of $(x_\alpha)_{\alpha\in A}$. Then we have 
$$
\mu\big(\{\alpha_u\leq\alpha\in A:(z-x_\alpha)-(z-x)\notin U\}\big)=0.
$$
So, we get $z\geq x$, i.e.,  $x_\alpha \uparrow x$ because of $z-x_\alpha\in E_+$ for all $\alpha\in A$.
\end{proof}

\begin{proposition}
If $x_\alpha\stc x$ and $y_\alpha\stc y$ in a Hausdorff locally solid Riesz space $(E,\tau)$ then $x_\alpha\ge y_\alpha$ for all $\alpha$ implies $x\ge y$.
\end{proposition}

\begin{proof}
Assume that $x_\alpha\stc x$, $y_\alpha\stc y$ and $y_\alpha\leq x_\alpha$ for all $\alpha$. Then, for any zero $\tau$-neighborhood $U$, there exists an index $\alpha_0$ such that $\mu\big(B_{\{\alpha_0,V\}}\{x_\alpha,x\}\big)=\mu\big(B_{\{\alpha_0,V\}}\{y_\alpha,y\}\big)=0$ holds. Thus, we have
$$
\mu\big(\{\alpha_u\leq\alpha\in A:(x_\alpha-y_\alpha)-(x-y)\notin V\}\big)=0.
$$
That is, $x_\alpha-y_\alpha\stc x-y$. So, we have $x-y\in E_+$ because of $x_\alpha-y_\alpha\in E_+$. Hence, we obtain the desired result.
\end{proof}

\begin{proposition}\label{riesz space}
The family of all $\mu$-statistically topological convergent nets $E_{\mu-st_\tau}$ is a Riesz space.
\end{proposition}

\begin{proof}
It follows from Theorem \ref{basic remarkasic properties of st convergence}, $E_{\mu-st_\tau}$ is a vector space. Now, consider an arbitrary element $x:=(x_\alpha)_{\alpha_\in A}$ in $E_{\mu-st_\tau}$ such that $x\stmuc y$ for some $y\in E$. Thus, by applying Theorem \ref{lattice operators}, we get $|x|\stmuc |y|$. It means that $|x|\in E_{\mu-st_\tau}$. Therefore, by using \cite[Thm.1.3 and Thm.1.7]{AB}, one can obtain that $E_{\mu-st_\tau}$ is a Riesz subspace.
\end{proof}

\begin{theorem}
Every uniformly continuous operator between locally solid Riesz spaces is a $\mu$-statistically continuous operator.
\end{theorem}

\begin{proof}
Suppose that $T:(E_\tau)\to(F,\tau')$ is a uniformly continuous operator and $(x_\alpha)_{\alpha\in A}$ $\mu$-statistically $\tau$-convergent to $x\in E$. Let $U$ be an arbitrary zero $\tau'$-neighborhood. Then there exists a zero $\tau$-neighborhood $V$ such that $T(v)\in U$ for every $v\in V$. Thus, $(x_\alpha-x)\in V$ implies $T(x_\alpha)-T(x)=T(x_\alpha-x)\in U$. It follows from $x_\alpha\stc x$ that there is an index $\alpha_v$ such that 
$$
\mu\big(\{\alpha_v\leq\alpha\in A:(x_\alpha-x)\notin V\}\big)=0.
$$
Then we observe that $(x_\alpha-x)\notin V$ whenever $T(x_\alpha)-T(x)\notin U$. Therefore, we obtain that
$$
\mu\big(\{\alpha_v\leq\alpha\in A:(T(x_\alpha)-T(x))\notin U\}\big)=0.
$$ 
As a result, $T(x_\alpha)\stc T(x)$, i.e., $T$ is $\mu$-statistically continuous.
\end{proof}

\end{document}